\documentclass{amsart}
\usepackage[T2A]{fontenc}
\usepackage[english]{babel}
\usepackage{amsmath}
\usepackage{amssymb}

\numberwithin{equation}{section}
\newtheorem{theorem}{Theorem}
\newtheorem{lem}{Lemma}
\newtheorem{pr}{Proposition}
\newtheorem{cory}{Corollary}
\theoremstyle{definition}
\newtheorem{df}{Definition}
\newtheorem{rem}{Remark}

\newcommand{\de}{\delta}
\newcommand{\er}{\varepsilon}

\newcommand{\si}{\sigma}
\newcommand{\sid}{\sigma_d}
\newcommand{\meal}{m}
\newcommand{\za}{\zeta}

\newcommand{\hol}{\mathcal{H}ol}

\newcommand{\Dbb}{\mathbb D}
\newcommand{\DDb}{\overline{\mathbb D}}
\newcommand{\Tbb}{\partial \mathbb{D}}

\newcommand{\sph}{\partial B}

\newcommand{\bd}{B_d}
\newcommand{\spd}{\partial B_d}
\newcommand{\ccd}{\mathbb{C}^d}
\newcommand{\bbd}{\overline{B}_d}
\newcommand{\MM}{M_+}
\newcommand{\qball}{Q}
\newcommand{\hp}{{H^p}}

\newcommand{\deker}{\mathcal{K}}
\newcommand{\bdd}{b}
\newcommand{\dbr}{\mathcal{H}}

\begin{document}

\title[Reverse Carleson measures]{Reverse Carleson measures for Hardy spaces
in the unit ball}
\author{Evgueni Doubtsov}
\address{Department of Mathematics and Computer Science,
St.~Petersburg State University,
Line 14th (Vasilyevsky Island), 29, St.~Petersburg 199178,
Russia
}
\email{dubtsov@pdmi.ras.ru}

\thanks{This research was supported by the Russian Science Foundation (grant No.~19-11-00058),
https://rscf.ru/project/19-11-00058/}


\begin{abstract}
Let $H^p=H^p(B_d)$ denote the Hardy space in the open unit ball
$B_d$ of $\mathbb{C}^d$, $d\ge 1$.
We characterize the reverse Carleson measures for $H^p$, $1<p<\infty$,
that is, we describe all finite positive Borel measures $\mu$,
defined on the closed ball $\overline{B}_d$, such that
\[
\|f \|_\hp \le c \|f\|_{L^p(\overline{B}_d,\mu)}
\]
for all $f\in H^p(B_d) \cap C(\overline{B}_d)$ and a universal constant $c>0$.
Given a non-inner holomorphic function $b: B_d \to B_1$, we obtain
properties of the reverse Carleson measures
for the de Branges--Rovnyak space $\mathcal{H}(b)$.
\end{abstract}

\maketitle

\section{Introduction}\label{s_int}
Let $\bd$ denote the open unit ball of $\ccd$, $d\ge 1$,
and let $\si=\sid$ denote the normalized Lebesgue measure on the unit sphere $\spd$.
For the unit disk $B_1$, the measure $\si_1$ and the unit circle $\sph_1$, we also use the
symbols $\Dbb$, $\meal$ and $\Tbb$, respectively.

Let $\hol(\bd)$ denote the space of holomorphic functions in the ball $\bd$.
For $0<p<\infty$, the Hardy space $H^p=H^p(\bd)$
consists of functions $f\in \hol(\bd)$ such that
\[
\|f\|_{H^p}^p = \sup_{0<r<1} \int_{\spd} |f(r\zeta)|^p\, d\si(\zeta) < \infty.
\]
As usual, the Hardy space $H^p(\bd)$, $p>0$,
is identified with the space
$H^p(\spd)$, which consists of the corresponding boundary values.
Further details about the spaces $\hp(\bd)$ and related objects are contained,
for example, in monographs \cite{Ru80} and \cite{Z05}.

\subsection{Carleson measures}
Let $\MM(\bd)$ denote the space of all finite
positive Borel measures defined on $\bd$.
Firstly, assume that $d=1$ and
$\mu\in \MM(\Dbb)$.
The classical Carleson Theorem \cite{Ca58} states that
$H^p(\Dbb)$ is embedded into $L^p(\Dbb, \mu)$, that is, there exists a constant $C > 0$ such that
\begin{equation}\label{e_Carl_1}
\| f \|_{L^p(\Dbb, \mu)} \le C \|f\|_\hp \quad\textrm{for all}\ f \in\hp
\end{equation}
if and only if $\mu$ satisfies the following Carleson condition:
there exists a constant $C > 0$ such that
\begin{equation}\label{e_Carl_geom_1}
\mu(S_I) \le C \meal(I)\quad\textrm{for all arcs}\ I\subset\Tbb,
\end{equation}
where
\[
S_I = \left\{z \in\overline{\Dbb}:\ 1-\meal(I) \le |z| \le 1,\ \frac{z}{|z|} \in  I \right\}.
\]

For an arbitrary dimension $d\ge 1$, consider the sets
\[
Q = Q(\za, \delta) =\{\xi\in \spd: |1-\langle \za, \xi\rangle | \le \delta\},
\quad \za\in\spd,\ \delta>0.
\]
Observe that $Q$ is a closed ball with respect to the nonisotropic metric
\[
\rho(\za, \xi) = |1-\langle \za, \xi\rangle |^{\frac{1}{2}}, \quad \za, \xi\in\spd.
\]
For $r>0$ and $Q = Q(\za, \delta)$, put $rQ = Q(\za, r\delta)$ by definition.

By the theorem on Carleson measures for the Hardy space in the unit ball
(see, for example, \cite[Ch.~5]{Z05}),
the estimate
\begin{equation}\label{e_Carl_n}
\| f \|_{L^p(\Dbb, \mu)} \le C \|f\|_\hp,\quad f \in\hp,
\end{equation}
is equivalent to the following condition on the measure $\mu$:
\begin{equation}\label{e_Carl_geom_n}
\mu(S_Q) \le C\si(Q)\quad\textrm{for all nonisotropic balls}\ Q\subset\spd,
\end{equation}
where
\[
S_Q = \left\{z \in\bd:\ 1-\si(Q) \le |z| \le 1,\ \frac{z}{|z|} \in  Q \right\}
\]
is a standard Carleson window.

\subsection{Reverse Carleson measures}
Remark that Luecking \cite{Lu85} studied
Carleson measures for Bergman spaces in the unit disk
together with those measures for which the corresponding inverse inequality holds.
In the case of several complex variables,
we mention the recent related results by Green and Wagner \cite{GW23}
about the dominating sets for Bergman spaces on general domains in $\ccd$.

In the first part of the present paper, we investigate the reverse Carleson inequality
\begin{equation}\label{e_revrese}
\|f \|_\hp \le C \|f\|_{L^p(\bbd,\mu)}, \quad f\in  C(\bbd) \cap H^p(\bd),\ 1 < p < \infty,
\end{equation}
for $\mu\in\MM(\bbd)$.

Remark that the above problem is formulated
under assumption that the measure $\mu$ is defined on the closed ball $\bbd$.
In fact, for $\mu\in\MM(\bd)$ this problem degenerates and has no solutions;
see Theorem~\ref{t_Carl_Hp} stated below.

Observe that the theorem about the Carleson measures for $\hp(\bd)$
also has a formulation under the assumption $\mu\in \MM(\bbd)$.
Indeed, in this case, condition~\eqref{e_Carl_n} makes sense for the functions $f$ in the set
$C(\bbd) \cap H^p(\bd)$, which is dense in $H^p(\bd)$.
Condition \eqref{e_Carl_geom_n} implies that the restriction $\mu_{|\spd}$
is absolutely continuous with respect to Lebesgue measure $\si$,
and the Radon--Nikodym derivative ${d\mu_{|\spd}}/{d\si}$ is bounded.

For $d=1$ and a measure $\mu\in\MM(\DDb)$ with property \eqref{e_Carl_geom_1},
Lef\`evre, Li et al.~\cite{LLQR12} proved that
property~\eqref{e_revrese} is equivalent to the following condition:
\[
\mu(S_I ) \ge C\meal(I)\quad\textrm{for all arcs}\ I \subset \Tbb.
\]
Further, this result was proved for all $\mu\in\MM(\DDb)$ in \cite{HMNO14},
where other equivalent properties have been also obtained.

The main result of the present paper about the reverse Carleson measures
for the Hardy space $H^p(\bd)$ with $d\ge 1$
is Theorem~\ref{t_Carl_Hp} formulated in the next section.

\subsection{Reverse Carleson measures for Hardy spaces in the unit ball}
Let $1 < p < \infty$. For $w \in \bd$, consider the Cauchy kernels
\[
k_w(z) = \frac{1}{(1-\langle z, w\rangle)^d},\quad z\in\bd,
\]
and the corresponding normalized functions
\[
K_w = \frac{k_w}{\|k_w\|_{\hp}}.
\]
Proposition~1.4.10 from \cite{Ru80} guarantees that
\[
\|k_w\|_p 	= (1-|w|^2)^{-\frac{d}{q}}, \quad\textrm{where}\ \frac{1}{p} + \frac{1}{q} = 1.
\]

For $d=1$, the following result was obtained in \cite{HMNO14}.

\begin{theorem}\label{t_Carl_Hp}
Let $1 < p < \infty$ and $\mu \in \MM(\bbd)$, $d\ge 1$.
Then the following properties are equivalent.
\begin{enumerate}
  \item[(i)] There exists a constant $C > 0$ such that
\[
\int_{\bbd}
|f|^p\,d\mu \ge C \|f\|_\hp^p\quad\textrm{for all}\ f \in C(\bbd) \cap H^p.
\]
  \item[(ii)] There exists a constant $C > 0$ such that
\[
\int_{\bbd}
|K_w|^p\,d\mu \ge C \quad\textrm{for all}\ w\in\bd.
\]
  \item[(iii)] There exists a constant $C > 0$ such that
\[
\mu(Q) \ge C \si(\qball)\quad\textrm{for all nonisotropic balls}\ Q \subset \partial\bd.
\]
\end{enumerate}
\end{theorem}

It is mentioned above that the problem on inverse Carleson measures for $\hp(\bd)$, $1<p<\infty$,
has no solutions for the measures $\mu$ defined on the open unit ball $\bd$.
Indeed, this fact follows immediately from property~(iii) in Theorem~\ref{t_Carl_Hp}.
Further comments and the proof of Theorem~\ref{t_Carl_Hp}
are given in Section~\ref{s_hp}.

\subsection{Reverse Carleson measures for de Branges--Rovnyak spa\-ces}
The second type of spaces studied in this paper is that of
de Branges--Rovnyak spaces in the unit ball.

Let $\bdd: \bd\to \Dbb$ be a holomorphic function.
Direct inspection shows that the function
\[
\deker^b(z, w) = \frac{1-b(z)\overline{b(w)}}{(1-\langle z, w\rangle)^d},\quad z,w\in \bd,
\]
has the properties of a reproducing kernel.
The reproducing kernel Hilbert space $\dbr(b)\subset H^2$
with kernel $\deker^b(z, w)$
is called a de Branges--Rovnyak space.
Further details about these spaces
in the unit disk are given, for example, in monograph \cite{Sa94}.

On the one hand, if $\|b\|_\infty < 1$, then the space $\dbr(b)$ coincides with $H^2$,
up to an equivalent norm.
On the other hand, the space $\dbr(b)$ has an explicit form for an inner function $b$.
Recall that a holomorphic function
$I:\bd \to \Dbb$ is called \textit{inner} if
$|I(\za)|=1$ for $\si_d$-almost all $\za\in\spd$.
Here and in what follows,
$I(\za)$ denotes the limit
$\lim_{r\to 1-} I(r\za)$, as usual.
It is known that the corresponding limit exists $\si_d$-almost everywhere.
Direct computations of the reproducing kernel for the space
$H^2 \ominus I H^2$ show that
$\dbr(b) = H^2 \ominus I H^2$.
If $I$ is an inner function in the unit disk $\Dbb$,
then $H^2(\Dbb) \ominus I H^2(\Dbb)$ is a classical model space.

A number of results on the reverse Carleson measures for the spaces $\dbr(b)$ in the unit disk
are obtained in \cite{BFGHR14}.
In that paper, for certain functions $b$, descriptions of the measures under consideration are obtained and
it is shown, in particular, that it suffices to check the reverse Carleson inequality
on the corresponding reproducing kernels.
For arbitrary spaces $\dbr(b)$
in the unit disk, it is known that such a reproducing kernel thesis does not hold
(see \cite{HMNO14}).

In the final Section~\ref{s_Hb} of the present paper, an analog
of property~(ii) from Theorem~\ref{t_Carl_Hp}
is used to deduce properties of the reverse Carleson
measures for the de Branges--Rovnyak spaces $\dbr(b)$
in the unit ball.
The results in this direction are obtained
for non-inner functions $b$
and are contained in Theorem~\ref{t_t2} and Corollary~\ref{c_noninner},
see Section~\ref{s_Hb}.

\subsection*{Notation}
As usual, $C>0$ denotes a constant whose value can change from line to line.

\section{Reverse Carleson measures for Hardy spaces}\label{s_hp}
In this section, Theorem~\ref{t_Carl_Hp} is proved
and one more description of the measures under investigation
is given in terms of the standard Carleson windows $S_Q$,
see Proposition~\ref{p_p4} below.

\subsection{Auxiliary results}
For $0<h\le 1$ and a nonisotropic ball $Q$, define the Carleson $h$-window $S_{Q,h}$ by the following equality:
\begin{equation}\label{e_hwind}
 S_{Q,h} = \left\{z \in\bd:\ 1 - h \le |z| \le 1,\ \frac{z}{|z|} \in Q \right\}.
\end{equation}
Let $\nu = \nu_d$ denote Lebesgue measures on the complex Euclidean space $\ccd$.

\begin{lem}\label{l_balay}
Let $\qball\subset \spd$
be a nonisotropic ball. For $0<h\le 1$, put
\begin{equation}\label{e_PhiDef}
\Phi_h(z)
= \frac{1}{h}
\int_{S_{\qball,h}} \frac{(1-|w|^2)^{pd-d}}{|1-\langle z, w \rangle|^{pd}}\,
d\nu(w), \quad z\in\bbd.
\end{equation}
Then
\begin{enumerate}
  \item[(a)] $\Phi_h(z) \le C$ for all $z\in \bbd$ and $0<h\le 1;$
  \item[(b)] $\lim_{h\to 0 }\Phi_h(z)= 0$ for all $z\in \bbd\setminus Q$.
  \end{enumerate}
\end{lem}
\begin{proof}
(a) It is clear that
\[
\Phi_h(z)
\le \frac{1}{h}
\int_{1-h\le |w|\le 1} \frac{(1-|w|^2)^{pd-d}}{|1-\langle z, w \rangle|^{pd}}\,
d\nu(w), \quad z\in\bd.
\]
Next, integrating in polar coordinates, we have
\[
\begin{split}
\Phi_h(z)
&\le C(d) \frac{1}{h}
\int_{1-h}^1 (1-r^2)^{pd-d}
\int_{\spd}
\frac{d\si(\xi)}{|1- r\langle z, \xi \rangle|^{pd}}\, dr
\\
&\le C(d) \frac{1}{h}
\int_{1-h}^1 (1-r^2)^{pd-d}
(1-r|z|)^{d-pd}\, dr
\\
&\le C(d)
\frac{1}{h}
\int_{1-h}^1\, dr
 \le C(d)
\end{split}
\]
by \cite[Proposition~1.4.10]{Ru80}.

(b) Assume that ${z} \notin  \qball$. Then there exists a $\de=\de(z)>0$ such that
$|1-\langle z, w \rangle|\ge\de$
for all $w\in S_{\qball,h}$,
$0 < h \le 1$. Hence,
\[
\begin{split}
   0 \le  \Phi_h(z)
&\le \frac{1}{h}
\int_{S_{\qball,h}}
\frac{(1-|w|^2)^{pd-d}}{\delta^{pd}}
d\nu(w)
\\
&\le C \frac{1}{h}
\frac{\si(\qball) h^{pd-d+1}}{\delta^{pd}}
\to 0\quad\textrm{as}\ h\to 0+.
\end{split}
\]
The proof of the lemma is finished.
\end{proof}

\subsection{Proof of Theorem~\ref{t_Carl_Hp}}
$(\mathrm{i})\Rightarrow (\mathrm{ii})$
This implication is trivial, since $\|K_w\|_\hp =1$ for $w\in\bd$.

$(\mathrm{ii})\Rightarrow (\mathrm{iii})$
Fix a nonisotropic ball $\qball\subset \spd$.
Integrating the estimate from condition~(ii) over the set
\[
 \widetilde{S}_{Q,h} = \left\{z \in\bd:\ 1 - h \le |z| < 1,\ \frac{z}{|z|} \in Q \right\},
\quad 0<h\le 1,
 \]
with respect to the measure $\nu$, we have
\[
\begin{aligned}
  C_{\mathrm{ii}}|\qball| h
&\le
\int_{\widetilde{S}_{\qball,h}}
\int_{\bbd} |K_w|^p\, d\mu\, d\nu(w) \\
&= \int_{\bbd}
\int_{\widetilde{S}_{\qball,h}}
\frac{(1-|w|^2)^{\frac{pd}{q}}}{|1-\langle z, w \rangle|^{pd}} \,d\nu(w)\, d\mu(z).
\end{aligned}
\]
The last expression does not change after replacing the set
${\widetilde{S}_{\qball,h}}$ by
 ${S}_{\qball,h}$ in the inner integral.
 Therefore, the above estimate has the following form in terms
 of the function $\Phi_h$ defined by identity~\eqref{e_PhiDef}:
\begin{equation}\label{e_bdd_Ph_ge}
\int_{\bbd}
\Phi_h(z)\, d\mu(z) \ge C_{\mathrm{ii}}\si(\qball).
\end{equation}
On the one hand, properties (a) and (b) from Lemma~\ref{l_balay} guarantee that
\begin{equation}\label{e_Ph_lim0}
 \lim_{h\to 0+}\int_{\bbd\setminus Q} \Phi_h(z)\, d\mu(z) =0.
\end{equation}
On the other hand, by property~(b) from Lemma~\ref{l_balay}, we have
\begin{equation}\label{e_Q_Ph_le}
 \int_{Q} \Phi_h(z)\, d\mu(z) \le C\mu(Q)\quad \textrm{for all}\ h>0.
\end{equation}
Combining properties \eqref{e_bdd_Ph_ge}--\eqref{e_Q_Ph_le}, we obtain
\[
C\mu(\qball ) \ge C_{\mathrm{ii}}\si(\qball),
\]
as required.

$(\mathrm{iii})\Rightarrow (\mathrm{i})$
Let
\[
\mu_{|\spd} = g \si + \mu^s_{|\spd},
\]
be the Lebesgue decomposition, where
$\mu^s_{|\spd} \bot \si$
and $g\in L^1(\si)$ denotes the corresponding Radon--Nikodym derivative.
Theorems 5.3.1 and 5.3.2 from monograph \cite{Ru80} guarantee that
\[
g(\za) = \lim_{\delta\to 0+} \frac{\mu_{|\spd}(Q(\za, \delta))}{\si(Q(\za, \delta))}
\]
for $\si$-almost all $\za\in\spd$.
Thus, condition~(iii) implies that $g \ge C>0$ $\si$-almost everywhere.
Therefore,
\[
\int_{\bbd}
|f|^p\, d\mu \ge
\int_{\spd}
|f|^p\,d\mu \ge
C \int_{\spd} |f|^p\,d\si = C \|f\|_\hp^p
\]
for all $f \in H^p \cap C(\bbd)$.
The proof of Theorem~\ref{t_Carl_Hp} is finished.

\subsection{Comments and remarks}
\begin{rem}
Equivalence of properties~(i) and (ii) in Theorem~\ref{t_Carl_Hp} means that
to prove the reverse Carleson inequality, it suffices to verify this inequality
on the reproducing kernels $k_w$, $w\in\bd$.
In other words, the reproducing kernel thesis holds for the problem
on reverse Carleson measures for $\hp$, $1<p<\infty$.
\end{rem}

\begin{rem}
On the one hand, property~(iii) guarantees that Lebesgue measure $\si$, defined on the sphere $\spd$,
is absolutely continuous with respect to the restriction $\mu_{|\spd}$.
On the other hand, as it is shown in the proof of Theorem~\ref{t_Carl_Hp},
the Radon--Nikodym derivative $d\mu^a_{|\spd}/d\si$ is separated
from zero by a positive constant.
Here $\mu^a_{|\spd}$ denotes
the absolutely continuous part (with respect to $\si$) of $\mu_{|\spd}$.
\end{rem}

The following assertion shows that condition~(iii) from Theorem~\ref{t_Carl_Hp}
is equivalent to its analog in terms of the sets $S_Q$
used to characterize the classical Carleson measures.

\begin{pr}\label{p_p4}
Equivalent properties from Theorem~\ref{t_Carl_Hp} hold if and only if
there exists a constant $C> 0$ such that
\begin{equation}\label{e_iv}
\mu(S_\qball) \ge C \si(\qball)\quad\textrm{for all nonisotropic balls}\ Q \subset \partial\bd.
\end{equation}
\end{pr}
\begin{proof}
Clearly, property~(iii) from Theorem~\ref{t_Carl_Hp} implies \eqref{e_iv}.
To prove the reverse implication, assume that property~\eqref{e_iv} holds.
Fix a $\delta>0$ and a nonisotropic ball $Q = Q(\za, \delta)\subset\spd$ for some point $\za\in\spd$.
For $0< h < \delta$, choose a maximal family of pairwise disjoint nonisotropic balls
$Q_j = Q_j(\za_j, h)\subset 2Q$, $j=1, \dots, J$.
Observe that $Q \subset\cup_{1\le j \le J} 2Q_j$.
Since the $h$-window $S_{Q_j, h}$ coincides with the set $S_{Q_j}$ for $j=1,\dots, J$,
condition~(iii) guarantees that
\begin{equation}\label{e_mu_si}
\begin{split}
   C_0(d)\mu(S_{Q,h}) =\mu(S_{2Q,h})
&\ge \mu \sum_{j=1}^{J} S_{Q_j,h}
 =\sum_{j=1}^{J}
 \mu(S_{Q_j,h})
 \\
&\overset{\mathrm{(iii)}}\ge
 C \sum_{j=1}^{J} \si(Q_j)
 = C(d)\sum_{j=1}^{J} \si(2Q_j)
 \\
&\ge  C(d) \si(Q).
\end{split}
\end{equation}
If $Q\subset G \subset\bbd$ and the set $G$ is relatively open in $\bbd$,
then $G\supset S_{\qball,h}$ for a sufficiently small parameter $h>0$.
Thus, applying \eqref{e_mu_si} and the outer regularity of the measure
 $\mu\in\MM(\bbd)$, we have
\[
\begin{split}
\mu(\qball )
&= \inf\{\mu(G): \qball\subset G,\ G \ \textrm{is open in}\ \bbd\}
\\
&\ge \inf_{h>0}\mu(S_{\qball,h})
 \ge C(d)\si(\qball),
\end{split}
\]
as required.
\end{proof}

\section{Reverse Carleson measures for de Branges--Rovnyak spaces}\label{s_Hb}
To define a reverse Carleson measure $\mu$ for $H^p(\bd)$,
we compute the norm
$\|f\|_{L^p(\overline{B}_d,\mu)}$
only for $f\in H^p(B_d) \cap C(\overline{B}_d)$, see \eqref{e_revrese}.
For the spaces $\dbr(\bdd)$, we apply a different approach based
on the following notion.

\begin{df}
Let $\mu\in\MM(\bbd)$. A function $f\in\hol(\bd)$
is called $\mu$-admissible if $f$ has radial limits $\mu$-almost everywhere
on the sphere $\spd$.
Given a holomorphic function $\bdd: \bd\to\Dbb$,
the set of all $\mu$-admissible functions in the space
$\dbr(b)$ is denoted by $\dbr_\mu(b)$.
\end{df}

Now, assume that $f\in \dbr_\mu(b)$.
On the support of the measure $\mu_{\left|\spd\right.}$, the function $f$
is defined via its radial limits.
As a result, the quantity $\|f\|_{L^2(\overline{B}_d,\mu)}$ takes
value in the infinite interval $[0,+\infty]$.

\begin{df}
Let $\bdd: \bd\to\Dbb$ be a holomorphic function.
A measure $\mu\in\MM(\bbd)$ is called a reverse Carleson measure for the space $\dbr(\bdd)$
if
 \[
\|f\|_{\dbr(b)} \le C ||f\|_{L^2(\bbd, \mu)}
 \]
for all $f\in \dbr_\mu(b)$.
\end{df}

The main result of the present section is the following theorem on properties of
the reverse Carleson measures.
For $d=1$, this result was obtained in \cite{BFGHR14}.

\begin{theorem}\label{t_t2}
Let $\mu\in\MM(\bbd)$ and $\bdd: \bd\to\Dbb$ be a $\mu$-admissible holomorphic function.
Put
\[
g = \frac{d\mu_{\left|\spd\right.}}{dm}.
\]
Assume that $\mu$ is a reverse Carleson measure for the space $\dbr(b)$.
Then
\begin{equation}\label{e_t2_ness}
1-|b(\za)|^2 \le C (1-|b(\za)|^2)^2 g(\za)
\end{equation}
for $\sid$-almost all points $\za\in\spd$ and a universal constant $C>0$.
\end{theorem}

On the one hand, if $|b(\za)|=1$, then condition \eqref{e_t2_ness} imposes no restriction
at the point $\za$.
In particular, this condition degenerates completely for
an inner function $b$.
On the other hand, condition~\eqref{e_t2_ness} naturally simplifies for $|b(\za)|<1$
and provides the following explicit information.

\begin{cory}\label{c_noninner}
Let $\mu\in\MM(\bbd)$ and $b: \bd\to \Dbb$ be a $\mu$-admissible non-inner holomorphic function.
Assume that $\mu$ is a reverse Carleson measure for the space $\dbr(b)$.
Then
\[
\frac{d\mu_{\left|\spd\right.}}{d\si} \neq \mathbf{0}
\]
and
\begin{equation}\label{e_1minusb}
 \int_{ \{\za\in\spd:\, |b(\za)| < 1\} }
\frac{1}{1-|b|}
\,d\si < \infty.
\end{equation}
\end{cory}
\begin{proof}
Firstly, by Theorem \ref{t_t2}, inequality~\eqref{e_t2_ness} holds,
therefore,
$g \neq \mathbf{0}$ because $b$ is not an inner function.
Secondly, property~\eqref{e_t2_ness} implies that
\[
\frac{1}{1-|b(\xi)|} \le C g(\xi)
\]
for $\si$-almost all $\xi\in \{\za\in\spd:\, |b(\za)| < 1\}$.
Since $g \in L^1(\spd, \si)$, we obtain the required property~\eqref{e_1minusb}.
\end{proof}

Remark that property~\eqref{e_1minusb} imposes a restriction on those functions $\bdd$ for which
the set of the reverse Carleson measures for $\dbr(\bdd)$ is not empty.

\begin{cory}\label{c_AB_exmpl}
There exists a holomorphic function $\bdd: \bd\to \Dbb$ with the following property:
if $\mu\in\MM(\bbd)$ and
$b$ is $\mu$-admissible, then $\mu$ is not a reverse Carleson measure for the space $\dbr(\bdd)$.
\end{cory}
\begin{proof}
Choose a positive function $\psi\in C(\spd)$ such that
$\psi(\za_0)=1$, $0< \psi(\za) <1$ for $\za\in\spd\setminus \{\za_0\}$ and
\[
 \int_{\spd}
\frac{1}{1-\psi}
\,d\si = \infty.
\]
By \cite[Corollary~2]{A82e},
there exists a holomorphic function $\bdd: \bd\to \Dbb$ such that
$|\bdd| = \psi$ $\si$-almost everywhere on the sphere $\spd$.
Therefore, if $\mu\in\MM(\bbd)$ and
$\bdd$ is $\mu$-admissible, then condition \eqref{e_1minusb} does not hold.
Hence, $\mu$ is not a reverse Carleson measure for $\dbr(\bdd)$ by Corollary~\ref{c_noninner}.
\end{proof}

Further, Corollary~\ref{c_noninner}
provides negative information on the sampling sequences for $\dbr(\bdd)$.
Recall that
$\{w_j\}_{j=1}^\infty \subset \bd$ is called a sampling sequence for $\dbr(\bdd)$
provided that there exist constants
$C_2\ge C_1>0$ such that
\[
C_1 \|f\|^2_{\dbr(\bdd)} \le
\sum_{j=1}^\infty
\frac{|f(w_j)|^2}{\ \|\deker_{w_j}\|^2_{\dbr(\bdd)}}
\le C_2 \|f\|^2_{\dbr(\bdd)}
\]
for all $f\in \dbr(\bdd)$.
For $d=1$, the following assertion is proved in \cite{BFGHR14}.

\begin{cory}\label{c_sampli}
Let $\bdd: \bd\to \Dbb$ be a non-inner holomorphic function.
Then there is no sampling sequence for the space $\dbr(\bdd)$.
\end{cory}
\begin{proof}
Assume that
$\{w_j\}_{j=1}^\infty \subset \bd$ is a sampling sequence for $\dbr(\bdd)$.
Then
\[
\mu = \sum_{j=1}^\infty \frac{\delta_{w_j}}{\|\deker_{w_j}\|^2_{\dbr(\bdd)}}
\]
is a reverse Carleson measure for $\dbr(\bdd)$.
This contradicts Corollary~\ref{c_noninner} because $d\mu_{|\spd}/d\si = \mathbf{0}$.
\end{proof}

Now, we proceed to the proof of Theorem~\ref{t_t2}.

\subsection{Auxiliary results}
Given a holomorphic function $F: \bd\to \Dbb$, recall that
the boundary function $F(\xi)$ is defined for $\si$-almost all $\xi\in\spd$
in terms of radial limits.
The following assertion refines Lemma~\ref{l_balay} for $p=2$.

\begin{lem}\label{l_r47}
Let $Q\subset \spd$ be a nonisotropic ball and $F: \bd\to \Dbb$ be a holomorphic function.
Then the limit
\[
\lim_{r\to 1-}
\int_Q |F(r\xi)| \frac{(1-r^2)^d}{|1- \langle z, r\xi\rangle|^{2d}} d\si(\xi)
\]
\begin{enumerate}
 \item[(a)] is equal to zero for $z\in\bbd\setminus Q$,
 \item[(b)] is equal $|F(z)|$ for $\si$-almost all $z$ in the interior of $Q$.
\end{enumerate}
\end{lem}
\begin{proof}
(a) Since $F$ is a bounded function,
it suffices to apply Lemma~\ref{l_balay}(b) for $p=2$.

(b) Let $\za\in\spd$. Observe that
\[
\lim_{r\to 1-} \int_{\spd\setminus\widetilde{Q}(\za, \tau)} \frac{(1-r^2)^d}{|1- \langle \za, r\xi\rangle|^{2d}} d\si(\xi) =0,
\]
where $\widetilde{Q}(\za, \tau)\subset\spd$ is a nonisotropic ball. Thus,
to prove property~(b), it suffices to verify that
\begin{equation}\label{e_b_sph}
\lim_{r\to 1-}
\int_{\spd} |F(r\xi)| \frac{(1-r^2)^d}{|1- \langle \za, r\xi\rangle|^{2d}} d\si(\xi) = |F(\za)|
\end{equation}
for $\si$-almost all points $\za\in\spd$.

Now, to prove property~\eqref{e_b_sph},
assume that $\za\in\spd$ is a Lebesgue point for the function $F(\xi)$, i.e.,
\begin{equation}\label{e_Lpoint}
  \lim_{\tau\to 0+}
\int_{\widetilde{Q}(\za, \tau)} |F(\xi) - F(\za)|\, d\si(\xi) = 0.
\end{equation}
Consider the following Poisson integral:
\[
u(w) = \int_{\spd} |F(\xi)-F(\za)| \frac{(1-|w|^2)^d}{|1- \langle w, \xi\rangle|^{2d}} d\si(\xi),\quad w\in\bd.
\]
Remark that $\lim_{r\to 1-} u(r\za) = 0$ by \eqref{e_Lpoint} and Theorem~5.4.8 from \cite{Ru80}.

The function $|F(w)-F(\za)|$, $w\in\bd$, is plurisubharmonic, hence,
\[
\begin{split}
   0
&\le \limsup_{r\to 1-}
 \int_{\spd} |F(r\xi) -F(\za)|\frac{(1-r^2)^d}{|1- \langle \za, r\xi\rangle|^{2d}}\, d\si(\xi) \\
&\le \limsup_{r\to 1-} \int_{\spd} u(r\xi) \frac{(1-r^2)^d}{|1- \langle r\za, \xi\rangle|^{2d}}\, d\si(\xi)
    =\limsup_{r\to 1-} u(r^2\za)  = 0.
\end{split}
\]
One has $\left||F(r\xi)| - |F(\za)|\right| \le |F(r\xi) -F(\za)|$, thus,
the above estimate implies the required property
\eqref{e_b_sph} at the point $\za\in\spd$ under consideration.
To finish the argument, recall that property~\eqref{e_Lpoint}
holds for $\si$-almost all $\za\in\spd$.
\end{proof}

\subsection{Proof of Theorem~\ref{t_t2}}
The function $\deker^b_w$, $w\in\bd$, is $\mu$-admissible,
since $b$ is $\mu$-admissible by the hypothesis of the theorem.
Hence, by the definition of a reverse Carleson measure, we have
\[
\|\deker^b_w\|^2_{\dbr(b)} \le C ||\deker^b_w\|^2_{L^2(\mu)}
\]
for all $w\in\bd$. Rewriting this inequality in explicit terms, we obtain
\begin{equation}\label{e_dataT1}
\frac{1-|b(w)|^2}{(1-|w|^2)^d}
\le C\int_{\bbd}
\left| \frac{1-b(z)\overline{b(w)}}{(1-\langle z, w \rangle)^d}
\right|^2
d\mu(z).
\end{equation}

Let
\[
\mu_{\left|\spd\right.} = g\sid + \eta
\]
be the Lebesgue decomposition, where
$\eta \bot \mu_{\left|\spd\right.}$ and $g$ denotes the Radon--Nikodym derivative
$\frac{d\mu_{\left|\spd\right.}}{d\sid}$.
Put
\[
\rho := \mu_{\left|\bd\right.} = \mu - g\si - \eta.
\]
In these terms, estimate~\eqref{e_dataT1} has the following form:
\begin{equation}\label{e_48}
\begin{split}
 C(1-|b(w)|^2)
&\le
\int_{\bd} \frac{(1-|w|^2)^d}{|1-\langle z, w \rangle|^{2d}}
|1-b(z)\overline{b(w)}|^2\, d\rho(z) \\
&\quad +
\int_{\spd}
\frac{(1-|w|^2)^d}{|1-\langle z, w \rangle|^{2d}}
|1-b(z)\overline{b(w)}|^2\, d\eta(z) \\
&\quad +
\int_{\spd}
\frac{(1-|w|^2)^d}{|1-\langle z, w \rangle|^{2d}}
|1-b(z)\overline{b(w)}|^2\, g(z)\,d\sid(z) \\
&\quad := J_1 + J_2 + J_3.
\end{split}
\end{equation}

Now, choose a point $\za\in \spd$ and consider
a non-isotropic ball $Q= (\za, \tau)\subset \spd$, $\tau>0$.
Recall that the Carleson $h$-window $S_{Q,h}$ for $0<h\le 1$ is defined by the following equality:
\[
 S_{Q,h} = \left\{z \in\bd:\ 1 - h \le |z| \le 1,\ \frac{z}{|z|} \in Q \right\}.
\]
We integrate both parts of inequality \eqref{e_48} over the set $S_{Q,h}$
with respect to the volume Lebesgue measure $\nu$,
divide by $h$ and study the behavior of the resulting expressions as $h\to 0+$.

Firstly, for the integral of the left-hand side of inequality \eqref{e_48}, after
integration in polar coordinates we have
\begin{equation}\label{e_Jleft}
\begin{split}
  C \int_{S_{Q,h}}
&(1-|b(w)|^2) \, d\nu(w) \\
&= \frac{C(d)}{h}
 \int_{1-h}^1
 \int_Q (1-|b(r\xi)|^2)\,d\si(\xi)\,r^{2d-1}dr \\
&\to C(d)\int_Q (1-|b(\xi)|^2)\,d\si(\xi)
\end{split}
\end{equation}
as $h \to 0$.

Secondly, applying Fubini's theorem, we obtain
\[
\begin{split}
   &\frac{C(d)}{h} \int_{S_{Q,h}} J_1(w)\, d\nu(w) \\
     &=\int_{\bd}
\frac{1}{h}
\int_{1-h}^{1}
\int_Q
\frac{(1-r^2)^d}{|1-\langle z, r\xi \rangle|^{2d}} |1-b(z)\overline{b(r\xi)}|^2\, d\si(\xi)
\,
r^{2d-1}dr\,d\rho(z).
\end{split}
\]
Remark that
\begin{equation}\label{e_b2_bdd}
\begin{split}
  \int_Q
 \frac{(1-r^2)^d}{|1-\langle z, r\xi \rangle|^{2d}}
&|1-b(z)\overline{b(r\xi)}|^2\, d\si(\xi) \\
&\le C \int_Q
\frac{(1-r^2)^d}{|1-\langle z, r\xi \rangle|^{2d}} \, d\si(\xi)
\le C
\end{split}
\end{equation}
for all $z\in\bbd$, and
\[
\frac{1}{h}
\int_{1-h}^{1}
\int_Q
\frac{(1-r^2)^d}{|1-\langle z, r\xi \rangle|^{2d}} \, d\si(\xi)\, r^{2d-1}dr
\overset{h \to 0}\longrightarrow 0\quad \textrm{for all}\ z\in\bd
\]
by Lemma~\ref{l_r47}(i) with $F\equiv 1$.
Therefore,
\begin{equation}\label{e_J1}
\frac{1}{h} \int_{S_{Q,h}} J_1(w)\, d\nu(w) \overset{h \to 0}\longrightarrow 0.
\end{equation}

Next,
\[
\frac{1}{h}
\int_{1-h}^{1}
\int_Q
\frac{(1-r^2)^d}{|1-\langle z, r\xi \rangle|^{2d}} \, d\si(\xi)\, r^{2d-1}dr
\overset{h \to 0}\longrightarrow 0\quad \textrm{for all}\ z\in\spd\setminus Q
\]
by Lemma~\ref{l_r47}(i) with $F\equiv 1$.
Therefore, applying Fubini's theorem and property~\eqref{e_b2_bdd}, we obtain
\begin{equation}\label{e_J2}
\begin{split}
   &\frac{C(d)}{h} \int_{S_{Q,h}} J_2(w)\, d\nu(w) \\
     &=\int_{\spd}
\frac{1}{h}
\int_{1-h}^{1}
\int_Q
\frac{(1-r^2)^d}{|1-\langle z, r\xi \rangle|^{2d}} |1-b(z)\overline{b(r\xi)}|^2\, d\si(\xi)
\,r^{2d-1}dr\, d\eta(z) \\
&\le
\int_Q
d\eta + \er(h),
\end{split}
\end{equation}
where  $\er(h) \to 0$ as $h\to 0$.

Finally, after application of Fubini's theorem, we have
\[
\begin{split}
   &\frac{C(d)}{h}
    \int_{S_{Q,h}} J_3(w)\, d\nu(w) \\
     &=\int_{\spd}
\frac{1}{h}
\int_{1-h}^{1}
\int_Q
\frac{(1-r^2)^d}{|1-\langle z, r\xi \rangle|^{2d}} |1-b(r\xi)\overline{b(z)}|^2 d\si(\xi)
r^{2d-1}dr g(\za)d\si(z).
\end{split}
\]
Applying Lemma~\ref{l_r47} with $F(w)= (1-b(w)\overline{b(\za)})^2$, we obtain
\[
C\ge
\int_{Q}
\frac{(1-r^2)^d}{|1- \langle z, r\xi\rangle|^{2d}} |1-b(r\xi)\overline{b(z)}|^2\, d\si(\xi)
\overset{r \to 1-}\longrightarrow
(1-|b(z)|^2)^2 \chi_Q(z)
\]
for $\si$-almost all $z\in\spd$.
Therefore,
\begin{equation}\label{e_J3}
\frac{1}{h}\int_{S(Q, h)} J_3(w)\, d\nu(w) \to
\int_Q
(1-|b|^2)^2 g\,d\si
\end{equation}
as $h \to 0$.

Recall that $Q = Q(\za, \tau)$.
So, combining properties \eqref{e_Jleft} and \eqref{e_J1}--\eqref{e_J3}, we obtain
\[
\int_{Q(\za, \tau)}
(1-|b|^2)\, d\si \le C
\int_{Q(\za, \tau)}
(1-|b|^2)^2 g\, d\si + C
\int_{Q(\za, \tau)} d\eta
\]
for all $\za\in\spd$.

Now, we divide both parts of the above estimate by $\si(Q(\za, \tau))$
and let $\tau$ tend to zero.
Since $\eta$ is a singular measure, Theorem~5.3.2 from monograph \cite{Ru80} guarantees that
\[
\frac{1}{\si(Q(\za, \tau))}\int_{Q(\za, \tau)} d\eta \overset{\tau\to 0}\longrightarrow 0
\]
for $\si$-almost all $\za\in\spd$.
Therefore, the required inequality \eqref{e_t2_ness} holds for $\si$-almost all $\za\in\spd$
by the measure differentiation theorem on the sphere $\spd$ (see \cite[Theorem~5.3.1]{Ru80}).
The proof of Theorem~\ref{t_t2} is finished.

\bibliographystyle{amsplain}

\begin{thebibliography}{1}

\bibitem{A82e}
A.~B. Aleksandrov, \emph{The existence of inner functions in the ball}, Mat.
  Sb. (N.S.) \textbf{118(160)} (1982), no.~2, 147--163 (Russian); English
  transl.: Math. USSR--Sb. \textbf{46} (1983), no.\,2, 143--159.

\bibitem{BFGHR14}
A.~Blandign\`eres, E.~Fricain, F.~Gaunard, A.~Hartmann, and W.~T. Ross,
  \emph{Direct and reverse {C}arleson measures for {$\mathcal{H}(b)$} spaces},
  Indiana Univ. Math. J. \textbf{64} (2015), no.~4, 1027--1057.

\bibitem{Ca58}
L.~Carleson, \emph{An interpolation problem for bounded analytic functions},
  Amer. J. Math. \textbf{80} (1958), 921--930.

\bibitem{GW23}
A.~W. Green and N.~A. Wagner, \emph{Dominating sets in {B}ergman spaces on
  strongly pseudoconvex domains}, Constr. Approx. (2023),
  https://doi.org/10.1007/s00365--023--09639--z.

\bibitem{HMNO14}
A.~Hartmann, X.~Massaneda, A.~Nicolau, and J.~Ortega-Cerd\`a, \emph{Reverse
  {C}arleson measures in {H}ardy spaces}, Collect. Math. \textbf{65} (2014),
  no.~3, 357--365.

\bibitem{LLQR12}
P.~Lef\`evre, D.~Li, H.~Queff\'{e}lec, and L.~Rodr\'{\i}guez-Piazza, \emph{Some
  revisited results about composition operators on {H}ardy spaces}, Rev. Mat.
  Iberoam. \textbf{28} (2012), no.~1, 57--76.

\bibitem{Lu85}
D.~H. Luecking, \emph{Forward and reverse {C}arleson inequalities for functions
  in {B}ergman spaces and their derivatives}, Amer. J. Math. \textbf{107}
  (1985), no.~1, 85--111.

\bibitem{Ru80}
W.~Rudin, \emph{Function theory in the unit ball of {${\bf C}^{n}$}},
  Grundlehren der Mathe\-matischen Wissenschaften, vol. 241, Springer-Verlag,
  New York-Berlin, 1980.

\bibitem{Sa94}
D.~Sarason, \emph{Sub-{H}ardy {H}ilbert spaces in the unit disk}, University of
  Arkansas Lecture Notes in the Mathematical Sciences, vol.~10, John Wiley \&
  Sons, Inc., New York, 1994, A Wiley-Interscience Publication.

\bibitem{Z05}
K.~Zhu, \emph{Spaces of holomorphic functions in the unit ball}, Graduate Texts
  in Mathematics, vol. 226, Springer-Verlag, New York, 2005.

\end{thebibliography}

\end{document}